\documentclass[12pt, psamsfonts,  draft]{amsart}

\usepackage{amssymb,amsfonts,amsmath,graphicx,lineno}
\usepackage{mathrsfs}
\usepackage{anysize}
\usepackage[colorlinks, citecolor=black, urlcolor=black, linkcolor=black] \newif\ifdraft\drafttrue

\usepackage{anysize}
\usepackage[colorlinks, citecolor=black, urlcolor=black, linkcolor=black]{hyperref}
\usepackage{tikz}
\usetikzlibrary{shapes, arrows}
\usepackage{pgfplots}
\tikzset{>= angle 60}

\usepackage{amsthm}
\usepackage{amsfonts}
\usepackage{epsfig}

\newtheorem{theorem}{Theorem}
\newtheorem{lemma}{Lemma}[section]
\newtheorem{proposition}{Proposition}

\theoremstyle{definition}
\newtheorem{definition}{Definition}

\def\eps{\varepsilon}

\def\R{{\mathbb R}}

\def\Q{{\mathbb Q}}

\def\N{{\mathbb N}}
\def\Z{{\mathbb Z}}

\def\bx{{\bfx}}

\def\lf{ \lfloor}
\def\rf{ \rfloor}
\def\lc{ \lceil}
\def\rc{ \rceil}

\def\frac  #1#2{\mkern 1mu {{#1} \over {#2}}\mkern 1mu }

\def \bfit {\bfseries \itshape }

\def \bfx {{\hbox{\bfit x}\mkern 1mu }} 
\def \bfy {{\hbox{\bfit y}\mkern 1mu }} 
\def \bfs {{\hbox{\bfit s}\mkern 1mu }}

\def \bfc {{\hbox{\bfit c}\mkern 1mu }}

\begin{document}

\title[On the Diophantine nature of the elements of Cantor sets]{On the 
Diophantine nature of the elements of Cantor sets arising in the 
dynamics of contracted rotations}

\author{ Yann Bugeaud, Dong Han Kim, Michel Laurent and Arnaldo Nogueira}

\thanks{DK was supported by the National Research Foundation of Korea (NRF-2018R1A2B6001624).}

\maketitle

\begin{abstract}
We prove that these Cantor sets are made up of transcendental numbers, apart from their endpoints $0$ 
and $1$, under some arithmetical assumptions on the data. 
To that purpose, we establish a criterion of linear independence over the field of algebraic 
numbers for  the  three numbers $1$, a characteristic Sturmian number, 
and an arbitrary Sturmian number with the same slope.
\end{abstract}

\footnote{\footnote \rm 2010 {\it Mathematics Subject Classification:}   
11J91,   37E05. }


\section{Introduction and results} 

A notorious open question in Diophantine approximation, formulated 
by Mahler \cite{Mah84}, is to decide whether 
the middle-third Cantor set 
$$
{\mathcal C}= \left\{ \sum_{k\ge 1} {x_k\over 3^k} ; \, x_k\in\{0,2\}, \,\forall k \ge 1 \right\}, 
$$
defined as the set of real numbers in $[0, 1]$
having no digit $1$ in their ternary expansion, contains irrational algebraic elements. 
Since one commonly believes that irrational algebraic numbers are normal to base $3$
(as to every other integer base), 
it is expected that the answer to Mahler's question is negative, 
but we are very far away from being able 
to prove this. Recall that ${\mathcal C}$ can as well be dynamically defined: 
it is the attractor of the iterated function system (IFS)
$\{\phi_0, \phi_1\}$, where
$$
\phi_0 (x) = {\frac {1}{3}} x, \quad
\phi_1 (x) = {\frac {1}{3}} x + \frac {2}{3}. 
$$
The main goal of the present paper is to exhibit a 
family of Cantor sets in $[0, 1]$ which also naturally arise in dynamics, 
precisely in the dynamics of contracted rotations, and 
for which we can prove that they contain no  algebraic elements, except $0$ and $1$. 

Throughout this paper, $\lf x \rf$ denotes the integer part,  
$\lc x \rc$ the upper integer part, and $\{ x \}$ the 
fractional part of a real number $x$. 

\begin{definition} 
Let $\lambda$ and $\delta$ be real numbers with $0< \lambda<1$ and $1-\lambda < \delta <1$. 
We call the map $f_{\lambda,\delta}$ defined by
$$
f=f_{\lambda,\delta}:x\in [0, 1) \mapsto \{ \lambda x+\delta\} 
$$
a  contracted rotation of $[0, 1)$. 
\end{definition}

In particular, the map $f_{\lambda,\delta}$ is a {\it 2-interval piecewise affine contraction} 
on the interval $I = [0, 1)$.  
Figure~1 shows the graph of $f_{\lambda,\delta}$.  

\begin{figure}[ht]
\begin{tikzpicture}
\draw (0,0) -- (5,0);
\draw (0,0) -- (0,5);
\draw[thick] (0,4) -- (2.5,5);
\draw[thick] (2.5,0) -- (5,1);
\node   at (-0.3,-0.3) {$0$};
\node at (2.5,-0.6) {$\frac{1-\delta}{\lambda}$};
\node   at (-0.5,5) {$1$};
\node   at (-0.5,4) {$\delta$};
\draw[dashed] (0,1) -- (5,1);
\node   at (-1.2,1) {$\lambda + \delta -1$};
\node   at (5,-0.4) {$1$};
\draw[dashed] (2.5,0) -- (2.5,5);
\end{tikzpicture}
\caption{ A plot of $f_{\lambda, \delta}: I \to I$}
\end{figure}

The following facts concerning the dynamics  of contracted rotations 
have been established in 
\cite{Bu, BuC, Ha, JaOb, LN, LNB}. Each map $f= f_{\lambda,\delta}$ has a {\it rotation number} 
$\theta=\theta_{\lambda,\delta}$ satisfying $0 < \theta < 1$. 
For any given $\lambda$ in $(0, 1)$, this rotation number is rational for almost all values of 
$\delta$ in $(1 - \lambda, 1)$. 
In particular, it has been proved in \cite{LN} that $\theta$ is rational when 
$\lambda$ and $\delta$ are 
both algebraic. On the other hand, 
if $ \theta $ takes an irrational value, then the closure of  the limit set  $\bigcap_{n\ge 1}f^n(I)$,
where $f^n$ stands for the $n$-th iterate  of $f$,   is a Cantor set $C= C_{\lambda,\delta}$.
Then, any orbit $(f^n(x))_{n\ge 0}$ asymptotically 
approaches of an orbit contained in $C$, as $n$ tends to 
infinity. 

Moreover, for any real number $\lambda$ and any irrational number $\theta$  with 
$0 < \lambda, \theta <1$, there exists one and only one value $\delta$ with $1-\lambda < \delta <1$ 
for which $\theta_{\lambda,\delta}= \theta$. This number $\delta$ is given by the series
\begin{equation}
\delta = \delta(\lambda,\theta) 
= (1-\lambda)\left(1+\sum_{k\ge 1} (\lf (k+1)\theta \rf -\lf k\theta \rf )\lambda^k \right).   
 \label{delta} 
 \end{equation}

We are concerned with the Diophantine nature of the elements of 
the Cantor set $C_{\lambda, \delta(\lambda, \theta)}$.  It is easily observed that 
the endpoints of $C_{\lambda, \delta(\lambda,\theta)}$ are $0$ and $1$. 
Whenever $\lambda$ is an arbitrary algebraic number 
and $\theta $ an arbitrary irrational number, we expect that $0$ and $1$ 
are the only algebraic numbers contained in  
$C_{\lambda, \delta(\lambda,\theta)}$.   
Our first theorem confirms this expectation when $\lambda$ is the reciprocal of an integer.

\begin{theorem}
Let $b$ be an integer with $b \ge 2$ and $\theta$ an irrational number  with $0 < \theta <1$. Put 
$\lambda = 1/b$ and $ \delta = \delta(\lambda,\theta)$. Then, any element of the 
Cantor set $C_{\lambda,\delta}$ 
which differs from its endpoints $0$ and $1$ is a transcendental number. 
\end{theorem}

A key tool for the proof of Theorem 1 is a result of independent interest on linear 
independence of Sturmian numbers of the same slope. 

Sturmian words are, by definition, the infinite aperiodic words of minimal complexity. 
They can be described as follows. 
Let $\theta$ and $\rho$ be real numbers with
$0 \le \theta, \rho < 1$ and $\theta$ irrational.
For $n \ge 0$, set
$$
s_n := \big\lfloor (n+1) \theta + \rho \big\rfloor -
\big\lfloor n \theta + \rho \big\rfloor,
\quad
s'_n := \big\lceil (n+1) \theta + \rho \big\rceil -
\big\lceil n \theta + \rho \big\rceil. 
$$
Then, the infinite words
$$
{\bfs}_{\theta, \rho} := s_0 s_1 s_2 \ldots,
\quad
{\bfs}'_{\theta, \rho} := s'_0 s'_1 s'_2 \ldots 
$$
are, respectively, the lower and upper Sturmian words with slope $\theta$
and intercept $\rho$. 
Observe that ${\bfs}_{\theta, 0}$ and ${\bfs}'_{\theta, 0}$
differ only by their first letter, thus, there exists an infinite word ${\bfc}_{\theta}$, called the 
characteristic Sturmian word with slope $\theta$,  such that 
$$
{\bfs}_{\theta, 0} = 0 {\bfc}_{\theta}, \quad  {\bfs}'_{\theta, 0} = 1 {\bfc}_{\theta}. 
$$ 
Classical references on Sturmian words include 
\cite[Chapter~6]{Fogg02} and \cite[Chapter~2]{Loth02}.
If ${\bf a}$ and ${\bf b}$ are two distinct symbols, then a Sturmian sequence over $\{{\bf a} , {\bf b}\}$ 
is obtained from a Sturmian sequence over $\{0, 1\}$ by replacing $0$ by ${\bf a}$ and $1$ 
by ${\bf b}$. 

Let $b$ be an integer with $b \ge 2$. 
For an infinite word ${\bfx} = x_1 x_2 x_3 \dots$  
over the alphabet $\{0, 1, \ldots , b-1\}$,
set 
$$
\xi_{\bfx} = \sum_{n=1}^\infty {x_n \over b^n}.     
$$
If ${\bfx}$ is a Sturmian word (of slope $\theta$ and intercept $\rho$) over 
two elements of the alphabet $\{0, 1, \ldots , b-1\}$, then we call 
$\xi_{\bfx}$ a Sturmian number (of slope $\theta$ and intercept $\rho$). 

Ferenczi and Mauduit \cite{FeMa97} established that every 
Sturmian number is transcendental. Their proof combines combinatorial 
properties of Sturmian words with a result from 
Diophantine approximation. Further progress has been made in the understanding 
of the combinatorial properties of Sturmian words, firstly by 
Berth\'e, Holton, and Zamboni \cite{BeHoZa06}, and subsequently by 
Bugeaud and Kim \cite{BuKim19}. 

A careful analysis of the auxiliary lemmas obtained in \cite{BuKim19} allows us to go one step 
further and to decide whether two Sturmian numbers of the same slope 
are  linearly independent. 

\begin{theorem} 
Let $b$ be an integer with $b \ge 2$. Let ${\bf a}, {\bf b}$ be distinct integers in $\{0, 1, \ldots , b-1\}$. 
Let $\theta$ and $\rho$ be real numbers with
$0 \le \theta , \rho < 1$ and $\theta$ irrational. 
Let $\xi_0$ be a real number whose $b$-ary expansion is  a characteristic 
Sturmian sequence of slope $\theta$ over $\{{\bf a}, {\bf b}\}$. 
Let $\xi_1$ be a real number whose $b$-ary expansion is a  
Sturmian sequence of slope $\theta$ and intercept $\rho$ over $\{{\bf a}, {\bf b}\}$. 
The real numbers $1, \xi_0, \xi_1$ are linearly independent over the field 
of algebraic numbers if and only if there does not exist an integer $j$ 
such that $\rho$
is congruent to $j \theta$ modulo one. 
\end{theorem}

The `only if' part is easy. Indeed, assume that $\theta$ in $(0, 1)$ is an irrational number, $\rho$
is a real number and $j, k$ are integers such that $\rho = j \theta + k$. Then, 
for any positive integer $n$, we have 
$$
\big\lfloor (n - j +1) \theta + \rho \big\rfloor -
\big\lfloor (n - j) \theta + \rho \big\rfloor
= \big\lfloor (n+1) \theta   \big\rfloor -
\big\lfloor n \theta   \big\rfloor,
$$
and similarly if $\lfloor \cdot \rfloor$ is replaced by $\lceil \cdot \rceil$. 
This shows that, in this case, $\xi_1$ can be written as a linear combination of $1$ and $\xi_0$ with 
rational coefficients. 

It may well happen that `linearly independent over the field 
of algebraic numbers' could be replaced by `algebraically independent' in the theorem above, but 
such a result seems to be by far out of reach.

\section{Combinatorial properties of Sturmian words}

Let $\theta = [0; a_1, a_2, \ldots ]$ be an irrational real number in $(0, 1)$ 
with partial quotients $a_1, a_2, \dots$ and convergents $p_k/q_k = [0; a_1, \ldots , a_k]$. 
The symbols ${\bf a}, {\bf b}$ are viewed as distinct elements of $\{0, 1, \ldots , b-1\}$. 

Let $(M_k)_{k \ge 1}$ be the sequence of finite words associated with $(a_i)_{i \ge 1}$, 
that is, defined by
\begin{equation}
M_0 = {\bf a},  \quad M_1 =  {\bf a}^{a_1 - 1} {\bf b}, \quad
M_{k} = (M_{k-1})^{a_{k}} M_{k-2}, \quad \hbox{for $k \ge 2$}. 
\label{M}
\end{equation}
Note that {\bf ab} is a suffix of $M_k$ for any odd integer $k \ge 3$, 
while {\bf ba} is a suffix of $M_k$ for any
even integer $k \ge 2$.
For $k \ge 2$, we denote by $M_k^{--}$ the word $M_k$ deprived of its two last letters and we set  $M^*_k= M_k^{--}{\bf ba}$ or  $M^*_k= M_k^{--}{\bf ab}$,  according whether $k$ is odd or even. Thus, $M_k^*$  is the word which differs from $M_k$ by its 
last two letters and only by these letters. 
For instance, $M_2 = ({\bf a}^{a_1 - 1} {\bf b})^{a_2} {\bf a}$   
and $M_2^* = ({\bf a}^{a_1 - 1} {\bf b})^{a_2 - 1} {\bf a}^{a_1} {\bf b}$.    

The commutation relation
\begin{equation}
M_{k-1}M_k^{--} = M_kM_{k-1}^{--},
\label{11}
\end{equation}
holds for any $k\ge3$, from which follows the formula
\begin{equation}
M_k = M_{k-1} M'_k = M'_k M^*_{k-1}, \label{10}
\end{equation}
where we have set
$$
M'_k = (M_{k-1})^{a_{k} -1} M_{k-2}.
$$    

Furthermore, the word $\lim_{k \to \infty} M_k = c_1 c_2 \ldots $ 
is the characteristic Sturmian word of slope $\theta$
and we set 
\begin{equation}
\xi_0 = \sum_{n=1}^\infty {c_n \over b^n}.     \label{12}    
\end{equation}

Throughout this paper, $|W|$ denotes the length (number of letters) of the finite word $W$. 
Note that $|M_k| = q_k$.  
We begin with an easy lemma. 

\begin{lemma}
Let $k \ge 3$ be an integer and  $V$ a finite word over $\{ {\bf a}, {\bf b} \}$. 
Put  
$$
{\bfx} = V M_k M_k \cdots, \quad  {\bfy} = V M'_k  (M_k)^\infty.   
$$ 
Then, there exists a real number $\eta_k$ with $|\eta_k| \le 1$  
and
$$
\xi_{\bfx} - \xi_{\bfy} = 
{(-1)^k({\bf b} - {\bf a})(b-1) \over b^{|V|+q_k}} +
 {(-1)^k({\bf b} - {\bf a})(b-1) + \eta_k \over b^{|V|+2q_k}}.   
$$
In particular, if $V = (M_k)^a M_{k-1}$ for some nonnegative integer $a$, 
then 
$$
{\bfy}= (M_k)^a M_{k-1}M'_k(M_k)^\infty = (M_k)^\infty
$$ 
and 
$$
\xi_{\bfx} - \xi_{\bfy} = 
{(-1)^k({\bf b} - {\bf a})(b-1) \over b^{(a+1)q_k+q_{k-1}}} +
 {(-1)^k({\bf b} - {\bf a})(b-1) + \eta_k \over b^{(a+2)q_k+q_{k-1}}}.      
$$   
\end{lemma}

\begin{proof} 
By \eqref{10}, the words $$
V M'_{k} M^*_{k-1} M'_k M^*_{k-1} \quad
\hbox{and} \quad V M'_k M_{k-1} M'_k M_{k-1}
$$
are prefixes of ${\bfx}$ and ${\bfy}$, respectively.
Since ${\bf ab}$ and ${\bf ba}$ (respectively, ${\bf ba}$ and ${\bf ab}$) are suffixes of 
$M_{k-1}$ and $M^*_{k-1}$ for even $k$ (respectively, for odd $k$), this completes the proof.
\end{proof} 

Using Lemma 7.2  of \cite{BuKim19}, we have the following proposition.  
This is the key tool for the proof of Theorem 3 below. 

\begin{proposition}    
Let $k$ be an integer with $k \ge 3$  
and ${\bfx}$ a Sturmian word over $\{ {\bf a}, {\bf b} \}$ of slope $\theta = [0; a_1, a_2, \ldots ]$. 
Then, there exist a uniquely determined    
non-empty suffix $U_k$ of $M_k M_{k+1} = (M_k)^{a_{k+1}+1} M_{k-1} $
and an integer $d_{k+1}$  
 such that 
$$
d_{k+1} \in \{a_{k+1}, a_{k+1} + 1 \}  
$$
and 
$$
{\bfx} = U_k (M_k)^{d_{k+1}} M_{k-1} M_k M_{k+1}^{--} \dots  
 = U_k (M_k)^{d_{k+1}} M_{k-1} M_k M_k \dots      
$$ 
In particular, we have $U_k =M_{k+1}$ when $\bfx=\bfc_\theta$ is the characteristic Sturmian word.
\end{proposition}

\begin{proof}    
 By \cite[Lemma 7.2]{BuKim19}, there exists a unique word $W$ satisfying   
\begin{enumerate}
\item[(i)] $\bfx = W M_{k+1} M_{k} M_{k+1}^{--} \dots$, where $W$ is a non-empty suffix of $M_{k+1}$,   or 
\item[(ii)] $\bfx = W M_{k} M_{k+1} M_{k} M_{k+1}^{--} \dots$, where $W$ is a non-empty suffix of $M_{k+1}$,   or
\item[(iii)] $\bfx = W M_{k+1} M_{k} M_{k+1}^{--} \dots$, where $W$ is a non-empty suffix of $M_{k}$.
\end{enumerate}
This corresponds exactly to $2q_{k+1} +q_{k}$ cases. 
Moreover, it was also shown in \cite[Lemma 7.2]{BuKim19} that the prefixes  
of length $2q_{k+1} +q_{k}-1$ corresponding to these $2q_{k+1} +q_{k}$ cases are all distinct.

For (i) and (ii), we put $U_k = W$. 
Then $U_k$ is a suffix of $M_k M_{k+1}$ and  
$$\bfx = U_k (M_k)^{a_{k+1}} M_{k-1} M_{k} M_{k+1}^{--}  
  \dots \quad \text{ for  (i)}  $$
and
$$\bfx = U_k (M_k)^{a_{k+1}+1} M_{k-1} M_{k} M_{k+1}^{--}  
   \dots \quad \text{ for (ii)}.$$
Note that the assumption $k \ge 3$ implies that $M_k$ is a prefix of $M_{k+1}^{--}$.

For (iii), we write  $\bfx = W M_{k+1} M_{k} M_{k+1}^{--} \dots$ 
where $W$ is a non-empty suffix of $M_{k}$.
Let $\bfx' =  M_{k+1} M_{k} M_{k+1}^{--} \dots  $ be such that $\bfx = W \bfx'$.  
Since $M_{k+1} M_{k} M_{k+1}^{--} = M_{k+1} M_{k+1} M_{k}^{--}$ by \eqref{11}, the 
Sturmian word $\bfx'$ satisfies (i) or (ii) with $W=M_{k+1}$.
Therefore, we have   
$$\bfx' = M_{k+1} M_{k+1} M_{k} M_{k+1}^{--} \dots$$   or 
$$\bfx' = M_{k+1} M_{k} M_{k+1} M_{k} M_{k+1}^{--} \dots,$$
depending on the $(2q_{k+1} +q_{k}-1)$-th letter of $\bfx'$. 

We choose $U_k = W M_{k+1}$. Then $U_k$ is a suffix of $M_k M_{k+1}$ and 
$$
\bfx = W \bfx' = W M_{k+1} M_{k+1} M_{k} M_{k+1}^{--} \dots = U_{k} (M_k)^{a_{k+1}} M_{k-1}  M_{k} M_{k+1}^{--}  
 \dots
$$  
or 
$$
\bfx = W \bfx' = W M_{k+1} M_{k} M_{k+1} M_{k} M_{k+1}^{--} \dots = U_{k} (M_k)^{a_{k+1}+1} M_{k-1} M_{k} M_{k+1}^{--}  
 \dots.
$$

For the characteristic Sturmian word $\bfc_\theta$, the recurrence relations \eqref{M} show that 
$$
U_{k}    
= M_{k+1} = M_k^{a_{k+1}}M_{k-1}
$$ 
and that 
$$
d_{k+1} = 
\begin{cases}
a_{k+1} \quad &{\rm if} \quad a_{k+2} \ge 2
\\
a_{k+1}+1 \quad &{\rm if} \quad a_{k+2}=1.
\end{cases}
$$
  This concludes the proof of the proposition. 
\end{proof}

Throughout the rest of this section, we let ${\bfx}$ be a Sturmian word and keep the notation of 
Proposition 1.

\begin{lemma}   
Set $c = ({\bf b} - {\bf a})(b-1)$.  
For a given $k\ge 3$,  
put ${\bfy} = U_k (M_k)^\infty$ and $\bfy' = U_k M'_k (M_k)^\infty$.
Then there exist $\delta_k$ and $\delta'_k$ in $(-1, 1)$ satisfying 
$$
\xi_{\bfx} - \xi_{\bfy} = 
{ (-1)^k c + \delta_k b^{-q_k+2} \over b^{|U_k|+(d_{k+1} 
 +1)q_k+q_{k-1}}}, \qquad
\xi_{\bfx} - \xi_{\bfy'}  = 
{ (-1)^k c + \delta'_k b^{-q_k+2}  \over b^{|U_k|+q_k}} .
$$
\end{lemma}

\begin{proof}   
Setting $V = U_k (M_k)^{d_{k+1}}  
 M_{k-1}$,  
it follows from Lemma 2.1 that  
 there exist $\eta_k$ with $|\eta_k| \le 1$ 
  and $\delta_k$ in $(-1, 1)$ such that 
$$
\begin{aligned}
\xi_{\bfx} - \xi_{\bfy} & = 
{(-1)^k({\bf b} - {\bf a})(b-1) \over b^{|U_k|+(d_{k+1}+1)q_k+q_{k-1}}} +
 {(-1)^k({\bf b} - {\bf a})(b-1) + \eta_k \over b^{|U_k|+(d_{k+1}+2)q_k+q_{k-1}}} 
 \\
& = { (-1)^k c + \delta_k b^{-q_k+2} \over b^{|U_k|+(d_{k+1}+1)q_k+q_{k-1}}}.   
\end{aligned}  
$$

Put $V = U_k$.     
Then, 
$$
 {\bfx} = U_k (M_k)^{d_{k+1}}  
  M_{k-1} M_k M_k \dots = V  M_k M_k \dots
 $$
since $d_{k+1} 
 \ge 1$ and $M_{k-1}M_k = M_kM_{k-1}^*$ by \eqref{11}.
It follows from Lemma 2.1 that 
there exist $\eta'_k$  with $|\eta'_k| \le 1$ 
and $\delta'_k$ in $(-1, 1)$ such that 
$$
\begin{aligned}
\xi_{\bfx} - \xi_{\bfy'}  & = 
{(-1)^k({\bf b} - {\bf a})(b-1) \over b^{|U_k|+q_k}} +
 {(-1)^k({\bf b} - {\bf a})(b-1) + \eta'_k \over b^{|U_k|+2q_k}}
   \\
& = { (-1)^k c + \delta'_k b^{-q_k+2} \over b^{|U_k|+q_k}}.   
\end{aligned}
$$

\end{proof} 

We deduce from \eqref{10} that the infinite word $\bfy' = U_k M'_k (M_k)^\infty$ 
is purely periodic with period length $q_k$.    
Note that  for some integers $m_k$, $m'_k$,    we have  
$$ \xi_{\bfy} =  {m_k \over b^{|U_k|} ( b^{q_k} -1)} \ \mbox{ and }\ 
 \xi_{\bfy'} =  {m'_k \over  b^{q_k} -1}.$$

Set $\eps = 1/10$  
and define 
$$
r_k := 
\begin{cases}
 |U_{k}|, \quad &{\rm if } \quad q_{k+1} >   (1+ \eps)(|U_k| + q_k),
 \\ 
 0, \quad &{\rm if } \quad q_{k+1} \le  (1+ \eps)(|U_k| + q_k).
 \end{cases}
 $$
and
$$
t_k := 
\begin{cases}
d_{k+1} 
 q_k + q_{k-1}, \quad &{\rm if } \quad q_{k+1} >   (1+ \eps)(|U_k| + q_k),
 \\ 
|U_k| , \quad &{\rm if } \quad q_{k+1} \le  (1+ \eps)(|U_k| + q_k).
 \end{cases}
 $$

Then for $k\ge 3$,   
there exist an integer $m_k$  
and a real number $\delta_k$ with $ |\delta_k | < 1$ such that 
\begin{equation}
\xi_{\bfx} - {m_k \over b^{r_k} ( b^{q_k} -1)} 
= { (-1)^k c + \delta_k b^{-q_k+2} \over b^{r_k + q_k + t_k }}. 
\label{30}
\end{equation} 
In particular, for the real number $\xi_0$ defined  in \eqref{12} from the characteristic Sturmian $\bfc_\theta$,
 we have $U_k = M_{k+1} = (M_k)^{a_{k+1}}M_{k-1}$ and 
there exist an integer $m^0_k$ 
and a real number $\delta^0_k$ with $ |\delta^0_k | < 1$ such that 
\begin{equation}
\xi_{0} - {m^0_k \over b^{q_k} -1} 
= { (-1)^k c + \delta^0_k b^{-q_k+2} \over b^{q_k+ q_{k+1}}}. 
\label{31}
\end{equation}

\begin{lemma}
Suppose that there are infinitely many indices $k$ with $a_k \ge 2$. 
Then, there are infinitely many indices $k$ such that
$$
q_{k+1}  > (1+\eps) (2+\eps) q_k.  
$$
\end{lemma}

\begin{proof} 
If there are infinitely many $k$ with $a_k \ge 3$, then the lemma   holds true since 
$(1+\eps) (2+\eps) < 3$. Otherwise, there exists $k_0$ such that  
$a_k \le 2$ for $k \ge k_0$. Then, $q_{k-1} > q_k /3$ for $k > k_0$. 
Consequently, if $a_{k+1} = 2$ for some $k > k_0 + 1$, then
$$
q_{k+1} = 2q_k + q_{k-1} > {7 \over 3} q_k 
 > (1+\eps) (2+\eps) q_k. 
$$
This completes the proof of the lemma. 
\end{proof} 

\begin{lemma}
If $a_k = 1$ for all large integers $k$, then
there are infinitely many indices $k$ such that $|U_k| > (1 + \eps )q_k.$
\end{lemma}

\begin{proof} 
Our assumption implies that $q_{k+1} > (1 + \eps) q_k$ for every large integer $k$. 
Suppose that $|U_k| \le (1 + \eps )q_k$ for some large integer $k$. 
Then, $U_k$ is a suffix of $M_{k+1}$.   

On the other hand, $U_{k+1}$ is a non-empty suffix of $ M_{k+1} M_{k+2}= M_{k+1}^2M_k$. Thus,  $U_{k+1}$ is either of the form $V M_{k+1} M_k$ or $V M_k$ for some  non-empty suffix $V$ of $M_{k+1}=M_kM_{k-1}$, 
or $U_{k+1}$ is equal to a non-empty suffix $W$ of $M_k$.   
For the first case, Proposition 1 at levels $k$ and $k+1$ reads
$$ 
 U_k (M_k)^{d_{k+1}} M_{k-1}\dots
= \bfx =V M_{k+1} M_k (M_{k+1})^{d_{k+2}} \dots = V M_k M_{k-1} \cdots 
$$ 
which implies that $V= U_k$. 
For the second case, we have  
$$ 
 U_k (M_k)^{d_{k+1}} M_{k-1}\dots = \bfx =V M_k (M_{k+1})^{d_{k+2}} \dots = V M_k M_k M_{k-1}  \cdots 
 $$ 
 which implies again that $V= U_k $.
For the third  case, Proposition 1 at level $k+1$, gives us the expression
$$
 \bfx =W M_{k+1}  M_{k+1} M_k M_{k+1} M_{k+1} \dots = W M_{k+1}  M_k M_{k-1} M_k M_k \cdots, 
 $$
 when $d_{k+2} = 2$, or 
$$ 
\begin{aligned}
\bfx  =W M_{k+1} M_k M_{k+1} M_{k+2}^{--}  
  \dots &= W M_{k+1}  M_k M_k M_{k-1} M_{k} M_{k-1} M_k^{--}  \cdots \\  
  &=  W M_{k+1}  M_k M_k M_{k-1} M_{k} M_{k} M_{k-1}^{--}  \cdots ,  
\end{aligned}
$$
when $d_{k+2}= 1$. Both expressions  imply that $U_k = WM_{k+1}$, which is impossible since $U_k$ is assumed to be a suffix of $M_{k+1}$. 

Hence, we conclude that either $U_{k+1} = U_k M_k$ or $U_{k+1} = U_k M_{k+1} M_k$. 
If $U_{k+1} = U_k M_{k+1} M_k$, then $|U_{k+1}| > (1 + \eps )q_{k+1}$ and the lemma is established.
Suppose that $U_{k+\ell} = U_{k +\ell-1} M_{k+\ell-1}$ for $\ell \ge 1$. 
Since 
$$ 
\lim_{\ell \to \infty}{ |U_{k+\ell}| \over q_{k+\ell}} 
= \lim_{\ell \to \infty} {|U_{k} | + q_{k} + \dots + q_{k + \ell -1} \over q_{k+\ell} } = {1 + \sqrt 5 \over 2},
$$  
we have $|U_{k+\ell}| > (1 + \eps )q_{k+\ell}$ for some integer $\ell$. 
This concludes the proof of the lemma. 
\end{proof} 


\begin{lemma}
For any Sturmian number $\xi_{\bfx}$, 
the set ${\mathcal K}$ of integers $k$ satisfying 
$$
t_k > q_k + \eps (r_k + q_k), \qquad q_{k+1} >  (1+ \eps) (r_k + q_k),    
$$
is infinite. Moreover, for any $k$ in ${\mathcal K}$, there exist 
integers $m^0_k, m_k$ 
 such that 
$$
\Bigl| \xi_{0} - {m^0_k \over b^{q_k} -1}  \Bigr|
\le { 1 \over b^{r_k + 2 q_k + \eps (r_k + q_k) - 2}}, \quad
\Bigl| \xi_{\bfx} - {m_k \over b^{r_k} ( b^{q_k} -1)} \Bigr|
\le { 1 \over b^{r_k + 2 q_k + \eps (r_k + q_k) - 2}}. 
$$
\end{lemma}

\begin{proof} 
Let $k$ be such that
\begin{equation}
q_{k+1} > (1+ \eps) (|U_k| + q_k). \quad 
\label{52}
\end{equation}
Then, we have
$$
r_k =  |U_{k}|, \quad 
t_k = d_{k+1} q_k + q_{k-1}. 
$$
Thus,
$$
t_k = d_{k+1} q_k + q_{k-1} \ge q_{k+1} > (1+ \eps) (r_k + q_k)  > q_k+ \eps(r_k+q_k)
$$
and
$$
q_{k+1}  >  (1+ \eps) (|U_k| + q_k) =  (1+ \eps) (r_k + q_k).
$$
Consequently, the lemma holds true when \eqref{52} holds for infinitely many $k$. 

Assume now that there exists $k_0$ such that $q_{k+1} \le  (1+ \eps)(|U_k| + q_k)$  for all $k > k_0$.
Then, for $k > k_0$, we have 
$$
r_k = 0, \qquad 
t_k = |U_k|.
$$
If $a_k = 1$ for all large $k$, then we have $q_{k+1} >  (1+ \eps) q_k$ for all large $k$  
and by Lemma 2.4, there are infinitely many $k$'s such that $t_k = |U_k| >  (1+ \eps) q_k$.  
The remaining case is when 
 there are infinitely many $k$ with $a_k \ge 2$. 
It then follows from 
Lemma 2.3 that there are infinitely many $k$ such that 
$$q_{k+1}  > (1+\eps) (2+\eps) q_k.$$
For such an integer $k$, we have
$$
t_k = |U_k| \ge {q_{k+1} \over 1+\eps} - q_k > (1+ \eps) q_k = q_k + \eps (r_k + q_k)
$$
and
$$
q_{k+1} > (1+\eps) (2+\eps) q_k > (1+\eps) q_k = (1 + \eps) (r_k + q_k). 
$$
This shows that the set ${\mathcal K}$ is infinite. 

Moreover, for any $k$ in ${\mathcal K}$, it follows from \eqref{30} and \eqref{31} that there exist 
integers $m^0_k, m_k$ 
such that 
$$
\Bigl| \xi_{0} - {m^0_k \over b^{q_k} -1}  \Bigr|
\le { |c| + 1 \over b^{q_k+ q_{k+1}}}
\le { 1 \over b^{r_k + 2 q_k + \eps (r_k + q_k) - 2}}, 
$$
$$
\Bigl| \xi_{\bfx} - {m_k \over b^{r_k} ( b^{q_k} -1)} \Bigr|
\le  { |c| + 1 \over b^{r_k + q_k + t_k }}
\le { 1 \over b^{r_k + 2 q_k + \eps (r_k + q_k) - 2}}. 
$$
This completes the proof of the lemma. 
\end{proof}

\section{Proof of Theorem 2}

We start by establishing a weaker version of Theorem 2, namely we consider 
linear independence over the set of rational numbers.  
Keep the notations of Section 2. 

\begin{theorem}
The real numbers $1, \xi_0, \xi_1 = \xi_{\bfx}$ are linearly independent over the rationals
if and only if there does not exist an integer $j$ 
such that the intercept of ${\bfx}$ is congruent 
to $j \theta$ modulo one. 
\end{theorem}

\begin{proof} 
Let $z_0, z_1$ be nonzero integers. 
For any $k$ in ${\mathcal K}$, it follows from Lemma 2.5 that
$$
\left\vert z_0 \xi_0 + z_1 \xi_1   
 - {z_0 m^0_k b^{r_k} + z_1 m_k \over b^{r_k} (b^{q_k} - 1)}  \right\vert
< (| z_0| + | z_1|) b^{2 - 11 (r_k + q_k) /10},    
$$
If $z_0 \xi_0 + z_1 \xi_1$ is rational, then there exist infinitely many indices $k$ such that 
$$ 
z_0 \xi_0 + z_1 \xi_1   
 = {z_0 m^0_k b^{r_k} + z_1 m_k \over b^{r_k} (b^{q_k} - 1)}, 
$$
that is, such that 
$$
b^{r_k + t_k } ( (-1)^k c z_0 + \delta^0_k z_0 b^{-q_k+2} )
+ b^{q_{k+1}} ((-1)^k c z_1 + \delta_k z_1 b^{-q_k+2} ) =0,
$$
by \eqref{30} and \eqref{31}. Since $b^{-q_k}$ tends to $0$ as $k$ tends to infinity, we have 
$$b^{r_k + t_k } z_0 + b^{q_{k+1}} z_1 =0,$$
when $k$ is sufficiently large. 
Thus, there is an integer $C$ such that for arbitrarily large $k$ 
$$
r_k + t_k = q_{k+1} + C. 
$$
Note that
$$ r_k +t_k= 
\begin{cases}
|U_k| + d_{k+1} q_k + q_{k-1}, \quad &{\rm if} \quad q_{k+1} > (1+ \eps) (|U_k| + q_k) , 
\\
|U_k|, \quad &{\rm if} \quad q_{k+1} \le(1+ \eps) (|U_k| +q_k) . 
\end{cases}
 $$
If $|U_k| + d_{k+1}q_k + q_{k-1} = q_{k+1} + C$
 for infinitely many $k$, then $d_{k+1} = a_{k+1}$ and $|U_k| = C$ for infinitely many $k$.
Otherwise, we deduce that 
$$
|U_k| = q_{k+1} + C 
$$
 for infinitely many $k$. We thus distinguish three cases.

If $|U_k| = p$ for some integer $p >0$  
and infinitely many $k$,  
then $U_k$ is the suffix of $M_{k-1}$ of length $p$.  
For arbitrary large $k$ we have ${\bfx} = U_k M_k \dots ,$   
where $|U_k| = p$ and $U_k$ is a suffix of $M_{k-1}$.     
Therefore, if we denote by $\sigma$ the left shift map, we find that $\sigma^p ( {\bfx} ) = M_k \dots $ 
for arbitrary large $k$.  Thus, $\sigma^p(\bfx)$ coincides with  the characteristic Sturmian word $\bfc_\theta$.   

If $|U_k| = q_{k+1} + p$ for some integer 
  $p \ge 0$ and infinitely many  $k$, then $U_k = V M_{k+1} $ 
for some $V$ where $|V| = p$ and $V$ is a suffix of $M_k$. 
Since ${\bfx} = U_k \dots = V M_{k+1} \dots $ for arbitrary large $k$,  
it follows that $\sigma^p ( {\bfx} ) = M_{k+1} \dots $ for arbitrary large $k$, 
which is also the characteristic Sturmian word $\bfc_\theta$. 

If $|U_k| = q_{k+1} - p$ for some integer 
 $p\ge 0$ and infinitely many $k$, 
then $V U_k = M_{k+1} $ for some $V$ where $|V| = p$ and $V$ is a prefix of $M_{k+1}$. 
Since $V {\bfx} = V U_k \dots = M_{k+1} \dots $ for arbitrary large $k$,  
 we now find the relation $\sigma^p ( {\bfc_\theta} ) = {\bfx} $.   

It follows 
that the intercept of ${\bfx}$ is congruent to $(\pm p +1)\theta$ modulo one. 

The other direction is clear.
\end{proof}

\medskip

We are in position to complete the proof of Theorem 2. 

By \cite{FeMa97}, we already know that $\xi_0$ and $\xi_1$ are transcendental. 
Let $\alpha_0$ and $\alpha_1$ be non-zero real algebraic numbers. 
Our strategy is to apply the Subspace Theorem \cite{Schm72,Schm80} 
to prove that the real number $\alpha_0 \xi_0 + \alpha_1 \xi_1$ is transcendental. 
We assume that it is algebraic and derive a contradiction. 

Set $C = b^2 (|\alpha_0| + |\alpha_1|)$. 
Let $k$ be in ${\mathcal K}$. 
Observe that, by Lemma 2.5, we have 
\begin{equation}
\Bigl| \alpha_0 \xi_0 + \alpha_1 \xi_1 - {\alpha_0 m_k^0 \over b^{q_k} - 1} 
- {\alpha_1 m_k \over b^{r_k} (b^{q_k } - 1)} \Bigr|
<  {C \over b^{r_k + 2 q_k + \eps (r_k + q_k)}}.    \label{55}
\end{equation}
Multiplying \eqref{55}  by $b^{r_k} (b^{q_k } - 1)$, we deduce that 
\begin{equation}
| (\alpha_0 \xi_0 + \alpha_1 \xi_1) b^{r_k + q_k} - (\alpha_0 \xi_0 + \alpha_1 \xi_1) b^{r_k} - \alpha_0 m_k^0 b^{r_k} - \alpha_1 m_k |
<  {C \over b^{q_k + \eps (r_k + q_k)}}.   \label{63} 
\end{equation}
Said differently, the linear form with algebraic coefficients
$$
L_4 (X_1, X_2, X_3, X_4) 
= (\alpha_0 \xi_0 + \alpha_1 \xi_1) X_1 - (\alpha_0 \xi_0 + \alpha_1 \xi_1) X_2 - \alpha_0 X_3 - \alpha_1 X_4
$$
takes small values at the integer quadruple
$
(b^{r_k + q_k}, b^{r_k}, m_k^0 b^{r_k}, m_k). 
$
Define 
$$
L_1 (X_1, X_2, X_3, X_4) = X_1, \quad 
L_2 (X_1, X_2, X_3, X_4) = X_2, \quad
L_3 (X_1, X_2, X_3, X_4) = X_3.
$$
Clearly, these four linear forms are linearly independent. 
For every prime divisor $\ell$ of $b$, set 
$$
L_{i,\ell} (X_1,X_2,X_3,X_4) = X_i, \quad 1\le i \le 4.
$$
Using the obvious estimate  $|m_k^0| \gg \ll b^{q_k}$ and \eqref{63}, observe 
that there exist $\delta > 0$ and infinitely many $k$ in ${\mathcal K}$ such that 
$$
\prod_{j = 1}^4 \prod_{\ell \mid b} \, |L_{j, \ell} (b^{r_k + q_k}, b^{r_k}, m_k^0 b^{r_k}, m_k)|_{\ell} \,
\prod_{j=1}^4 |L_j (b^{r_k + q_k}, b^{r_k}, m_k^0 b^{r_k}, m_k) | \le H_k^{- \delta},     
$$
where $| . |_\ell$ stands for the usual $\ell$-adic absolute value on $\Q$, and where we have set
$$
H_k = \max\{b^{r_k + q_k}, b^{r_k}, |m_k^0| b^{r_k}, |m_k| \}.
$$
Then,  the Subspace Theorem \cite{Schm72,Schm80} asserts that
there exists a proper rational subspace of $\Q^4$ 
containing infinitely many quadruples $(b^{r_k + q_k}, b^{r_k}, m_k^0 b^{r_k}, m_k)$ 
with $k$ in ${\mathcal K}$. In other words,  
 there exist integers $y_1, \ldots , y_4$, not all zero, 
and an infinite subset ${\mathcal K}_0$ of ${\mathcal K}$ such that 
\begin{equation}
y_1 b^{r_k + q_k} + y_2 b^{r_k} + y_3 m_k^0 b^{r_k} + y_4 m_k = 0,  \label{5} 
\end{equation}
for every $k$ in ${\mathcal K}_0$. 
Dividing \eqref{5} by $b^{r_k} (b^{q_k} - 1)$ and letting $k$ tend to infinity along ${\mathcal K}_0$, we get
$$
y_1 + y_3 \xi_0 + y_4 \xi_1 = 0.
$$
This shows that 1, $\xi_0$ and $\xi_1$ are linearly dependent over $\Q$, a contradiction 
with Theorem 3.   
Consequently, $\alpha_0 \xi_0 + \alpha_1 \xi_1$ is transcendental.

\section{Proof of Theorem 1}

Following \cite{Co, LN, LNB}, we  introduce the following function  $\phi$. 

\begin{definition}
Let $\lambda, \theta$ be real numbers in $(0, 1)$. 
Let $\phi : \R \mapsto \R$ be the real function of the real variable $y$ defined by the formula
$$
\phi(y) =\phi_{\lambda,\theta}(y)=  {\delta(\lambda,\theta)\over 1-\lambda}  +  (1-\lambda)
\sum_{k\ge 0} \lf y- (k+1) \theta\rf  \lambda^k,
$$
where the value  $\delta(\lambda,\theta)$ is given  by \eqref{delta}.   
 \end{definition}
A detailled study of the  function $\phi$ may be found in \cite{LNB}. 
It satisfies the functional equation
$$
\{ \phi(y+ \theta)\} = f\left( \{ \phi(y)\}\right) , \quad \forall y \in \R.
$$
We deal here with the special case $\mu=1$ in the setting  of \cite{LNB}. 

The function $\phi$ enables us to   parametrize the  points of our Cantor set 
$C_{\lambda,\delta(\lambda,\theta)}$ introduced in Section 1. 
In the sequel, we assume that $\theta$ is an irrational number. 
Observe that $\phi$ is an increasing right continuous function which has a left discontinuity 
only at the points $l \theta  + \Z$ for any positive integer $l$. 
We denote by $\phi(\{ l \theta\}^-) $ the left limit of $\phi$ at the point $\{l \theta\}$ for $l\ge 1$. 

\begin{lemma}
For any real number $\lambda$ and any irrational $\theta$ with $0 < \lambda , \theta <1$, 
we have the equality of sets
$$
C_{\lambda,\delta(\lambda,\theta)} = \overline{ \phi(I)} 
= \phi([0,1])  \, \cup \, \bigcup_{l\ge 1} \phi(\{ l\theta\}^-).
$$
Moreover, the explicit formulae
$$
\begin{aligned}
\phi(\{ l \theta\})  &={(1-\lambda^l)\delta(\lambda,\theta)\over 1-\lambda} 
-  \sum_{k=1}^l \lambda^{l-k}\left(  \lf k\theta \rf - \lf (k-1)\theta\rf \right),
\\
\phi(\{ l \theta\}^-)  & =\phi(\{ l \theta\})  - \lambda^{l-1}(1-\lambda),
\\
\phi(\{- l \theta\}) & = {1-\lambda^{-l}\over 1 -\lambda} \delta(\lambda,\theta) 
+ \lambda^{-l}+\sum_{k=1}^{l-1} 
\lambda^{k-l}\big( \lf (k+1)\theta\rf - \lf k\theta\rf\big),
\end{aligned}
$$
hold for any integer $l\ge 1$.
\end{lemma}

\begin{proof}
It is established in Proposition 5 of \cite{LNB} that  
$$
\bigcap_{n\ge1}f^n(I)=  \phi(I) = I \setminus\bigcup_{l\ge  1}\big[\phi(\{ l \theta\}^-),\phi(\{ l \theta\})\big).
$$
We deduce that
$$
C_{\lambda,\delta(\lambda,\theta)} =\overline{ \phi(I)} 
= [0,1] \setminus\bigcup_{l\ge  1}\big(\phi(\{ l \theta\}^-),\phi(\{ l \theta\})\big)= 
\phi([0,1]) \, \cup \, \bigcup_{l\ge 1} \phi(\{ l\theta\}^-),
$$
since  $\phi$ is an increasing function with left discontinuities in $I$ at the points $\{ l\theta\}, \, l\ge 1$. 
The  formulae for $\phi(\{ l \theta\})$ and $\phi(\{ l \theta\}^-)$ arise from  Proposition 5 of \cite{LNB}.  
For $\phi(\{- l \theta\})$, write
$$
\begin{aligned}
\phi(\{- l \theta\})  & = {\delta(\lambda,\theta)\over 1-\lambda}  +  (1-\lambda)\sum_{k\ge 0} 
\lf \{-l\theta\} - (k+1) \theta\rf  \lambda^k
\\
=& {\delta(\lambda,\theta)\over 1-\lambda}  +  \sum_{k\ge 0} \big(\lf \{-l\theta\} - (k+1) \theta\rf  - 
\lf \{-l\theta\} - k \theta\rf \big)  \lambda^k
\\
=& {\delta(\lambda,\theta)\over 1-\lambda}  +  \sum_{k\ge 0} \big(\lf  - (k+l+1) \theta\rf  
- \lf -(k+l) \theta\rf \big)  \lambda^k
\\
=& {\delta(\lambda,\theta)\over 1-\lambda}  +  \lambda^{-l}\sum_{k\ge l} \big(\lf  - (k+1) \theta\rf  
- \lf -k \theta\rf \big)  \lambda^k
\\
=&{\delta(\lambda,\theta)\over 1-\lambda}  -  \lambda^{-l}\sum_{k\ge l} \big(\lf   (k+1) \theta\rf  
- \lf k \theta\rf \big)  \lambda^k
\\
=& {\delta(\lambda,\theta)\over 1-\lambda}  -  \lambda^{-l}\sum_{k\ge 1} \big(\lf   (k+1) \theta\rf  
- \lf k \theta\rf \big)  \lambda^k + \lambda^{-l}\sum_{k= 1}^{l-1} \big(\lf   (k+1) \theta\rf  
- \lf k \theta\rf \big)  \lambda^k
\\
=& {\delta(\lambda,\theta)\over 1-\lambda} - \lambda^{-l} 
\left( {\delta(\lambda,\theta)\over 1-\lambda}- 1\right) 
+\sum_{k= 1}^{l-1} \big(\lf   (k+1) \theta\rf  - \lf k \theta\rf \big)  \lambda^{k-l}
\\
=& {1 -\lambda^{-l}\over 1-\lambda} \delta(\lambda,\theta) + \lambda^{-l} + 
\sum_{k= 1}^{l-1} \big(\lf   (k+1) \theta\rf  - \lf k \theta\rf \big)  \lambda^{k-l}.
\end{aligned} 
$$
This completes the proof of the lemma. 
\end{proof}

We are now able to apply Theorem 2. Lemma 4.1 tells us that an arbitrary 
element $z$ in $C_{\lambda,\delta(\lambda,\theta)}$ can be written either in  the form  $z = \phi(y)$ 
for some real number $y$ with $0 \le y \le 1$, or  $z =\phi(\{ l \theta\}^-)$ for some integer $l\ge 1$. 
For any $y$ in $\R$, set 
$$
\bx_y = \big( \lc y+ (k+1)\theta \rc - \lc y+k\theta \rc\big)_{k\ge 1} \in \{ 0, 1\}^\N.
$$
The sequence $\bx_y$ is Sturmian with slope $\theta$ and intercept $y+\theta$ modulo one. 
Moreover, $\bx_0=\bfc_\theta$ is the characteristic Sturmian sequence of slope $\theta$ over $\{0,1\}$. Denote by
$$
\xi_{y, \lambda} = \sum_{k\ge 1} (\lc y+ (k+1)\theta \rc -\lc y+ k\theta \rc )\lambda^k
$$
the value at the point $ \lambda$ of the power series whose sequence of coefficients is equal to $\bx_y$. 
Then, we can write 
$$
\delta(\lambda,\theta) = (1-\lambda) (1+ \xi_{0, \lambda}),
$$
noting that  $\lc x \rc = \lf x \rf +1$ for any $x$ in $\R \setminus \Z$, and
$$
\begin{aligned}
\phi(y) =  \phi_{\lambda,\theta}(y)
= & {\delta(\lambda,\theta)\over 1-\lambda}  +  (1-\lambda)\sum_{k\ge 0} 
\lf y- (k+1) \theta\rf  \lambda^k.
\\
= & {\delta(\lambda,\theta)\over 1-\lambda}   + \lf y- \theta \rf + \sum_{k\ge 1}\big( \lf y-(k+1)\theta \rf - 
\lf y- k \theta \rf\big)\lambda^k
\\
=& {\delta(\lambda,\theta)\over 1-\lambda}  
 + \lf y- \theta \rf - \sum_{k\ge 1}\big( \lc -y+ (k+1)\theta \rc- \lc -y
+ k \theta \rc\big)\lambda^k
\\
=& 1 +\xi_{0, \lambda} +  \lf y-\theta \rf  - \xi_{-y, \lambda},
\end{aligned}
$$
since $\lf x \rf = -\lc - x \rc $ for any $x$ in $\R$.  We immediately deduce from the above formulae that
$$
\phi(0) = 0 \quad{\rm and} \quad \phi(1)= 1.
$$
When $\lambda= 1/b$ for an integer $b \ge 2$  and $y$ is not congruent modulo one
to some integer multiple $m\theta$ of $\theta$, 
the three numbers  $1,\xi_{0, \lambda}, \xi_{-y, \lambda}$ are linearly independent 
over the field of algebraic numbers by Theorem 2. Therefore, the non-zero linear combination
$$
\phi(y)= 1 + \xi_{0, \lambda} + \lf y-\theta \rf  - \xi_{-y, \lambda}
$$
is a transcendental number. It remains to deal with points $y$  in $(0,1)\cap (\Z \theta + \Z)$. 
These points have the form $y = \{ m\theta\}$,  for some  non-zero integer $m$. 
Observe that, by a result of \cite{LN} quoted in Section 1, 
the real number $\delta(1/b, \theta)$ is transcendental for any integer $b\ge 2$ 
and any irrational number $\theta$. Using the explicit formulae from the lemma, write
$$
\phi(\{ m\theta\})= {1-b^{-m}\over 1-b^{-1}} \, \delta(1/b,\theta)+ A_m, 
$$
where 
$$
A_m =\begin{cases}   -  \sum_{k=1}^m b^{k-m}\left(  \lf k\theta \rf - \lf (k-1)\theta\rf \right),
\quad  {\rm if} \quad  m \quad { \rm is \, \, positive}, 
\\
b^{-m}+\sum_{k=1}^{-m-1} b^{-k-m}\big( \lf (k+1)\theta\rf - \lf k\theta\rf\big),     \quad   {\rm if } 
\quad m \quad {\rm is \,\, negative},
\end{cases} 
$$
is a rational number. It follows that $\phi(\{ m\theta\})$ is a transcendental number, 
since the coefficient $(1-b^{-m})/(1-b^{-1})$  is a non-zero rational number. 
 Finally, $\phi(\{l \theta\}^-) = \phi(\{ l\theta\}) -\lambda^{l-1}(1-\lambda)$ is also transcendental. 
 Theorem 1 is proved.

\vskip 2cm

{\footnotesize

YANN BUGEAUD (yann.bugeaud@math.unistra.fr) :  Universit\'e de Strasbourg, IRMA, CNRS, UMR 7501, 7 rue Ren\'e Descartes, 67084, Strasbourg, France.

DONG HAN KIM (kim2010@dongguk.edu) : Department of Mathematics Education, Dongguk University - Seoul, 30 Pildong-ro 1-gil, Jung-gu, Seoul, 04620 Korea.

MICHEL LAURENT (michel-julien.laurent@univ-amu.fr) : Aix-Marseille Universit\'e, CNRS, Centrale Marseille, Institut de Math\'ematiques de Marseille, 163 avenue de Luminy, Case 907, 13288, Marseille C\'edex 9, France.

ARNALDO NOGUEIRA  (arnaldo.nogueira@univ-amu.fr) : Aix-Marseille Universit\'e, CNRS, Centrale Marseille, Institut de Math\'ematiques de Marseille, 163 avenue de Luminy, Case 907, 13288, Marseille C\'edex 9, France. 
}

\end{document}